\documentclass[11pt]{amsart}

\usepackage[foot]{amsaddr}
\usepackage{color,graphicx}
\usepackage{dsfont}
\usepackage{amssymb}
\usepackage{amsmath}
\usepackage{amsfonts}
\usepackage{graphicx}
\usepackage{amsthm}
\usepackage{enumerate}
\usepackage[mathscr]{eucal}
\usepackage{mathrsfs}
\usepackage{verbatim}
\usepackage{yhmath}

\calclayout

\usepackage{hyperref}
\hypersetup{colorlinks, linkcolor=blue}
\usepackage{soul}
\soulregister\cite7
\soulregister\eqref7
\soulregister\eqref7
\hypersetup{colorlinks, linkcolor=blue}

\newtheorem{thm}{Theorem}[section]
\newtheorem{lemma}[thm]{Lemma}
\newtheorem{prop}[thm]{Proposition}
\newtheorem{coro}[thm]{Corollary}

\theoremstyle{definition}

\newtheorem{remark}[thm]{Remark}

\newcommand\ddd{\mathrm{d}}
\newcommand\dist{\mathrm{dist}}

\newcommand\supp{\mathrm{supp}}

\newcommand\bR{\mathbb{R}}

\newcommand\bZ{\mathbb{Z}}

\newcommand\bT{\mathbb{T}}

\def \l {\left}
\def \r {\right}

\makeatletter
\@namedef{subjclassname@2020}{
	\textup{2020} Mathematics Subject Classification}
\makeatother

\allowdisplaybreaks[4]

\begin{document}
\title[Bilinear estimate for Schr\"odinger equation on $\bR \times \bT$]{Bilinear estimate for Schr\"odinger equation on $\bR \times \bT$}
\author{Yangkendi Deng}
\author{Boning Di}
\author{Chenjie Fan}
\author{Zehua Zhao}

\date{}

\begin{abstract}
We continue our study of bilinear estimates on waveguide $\mathbb{R}\times \mathbb{T}$ started in \cite{DFYZZ2024,Deng2023}. The main point of the current article is, comparing to previous work \cite{Deng2023}, that we obtain estimates beyond the semiclassical time regime. Our estimate is sharp in the sense that one can construct examples which saturate this estimate.
\end{abstract}

\maketitle

\tableofcontents

\section{Introduction}
\subsection{Statement of main results}
In this paper, we consider (linear) Schr\"odinger equation on waveguide $\mathbb{R}\times \mathbb{T}$ as follows,
\begin{equation}\label{eq: lsrt}
\begin{cases}
iu_{t}-\Delta u=0,\\
u(x, 0)=f(x).
\end{cases}
\end{equation}
We use $e^{it\Delta_{\mathbb{R}\times \mathbb{T}}}$ to denote the linear (Schr\"odinger) propagator. When there is no confusion, we short it as $e^{it\Delta}$ for convenience.

Note that the physical space here is $\mathbb{R}\times \mathbb{T}$, thus the corresponding (Fourier) frequency space is $\mathbb{R}\times \mathbb{Z}$. We use $Q$ to denote cubes on $\mathbb{R}\times \mathbb{Z}$, and $P_{Q}$ to the associated Littlewood-Paley projections\footnote{See Notations \ref{SubS:Notations} for more details.}.

The main result of this paper reads,
\begin{thm}\label{thm: main}
Consider equation \eqref{eq: lsrt}. Let $p\in (\frac{5}{3},2)$. Let $Q_{1}, Q_{2}$ be two cubes of length $R\ge 1$ on  $\mathbb{R}\times \mathbb{Z}$, and we assume they are $R$ separated, i.e. their centers are of distance $\sim R$. Let $T_{p}$ be defined as $R^{\frac{4p-8}{3-p}}$. Then for all $\varepsilon>0$, one has for all $f,g\in L^{2}(\mathbb{R}\times \mathbb{T})$,
\begin{equation}\label{eq: mainestimate}
\|e^{it\Delta}P_{Q_{1}}fe^{it\Delta}P_{Q_{2}}g\|_{L_{x,t}^{p}(\mathbb{R}\times \mathbb{T}\times [0,T_{p}R^{-\varepsilon}])}\lesssim_{\varepsilon} R^{\frac{2p-4}{p}}\|f\|_{L_{x}^{2}}\|g\|_{L_{x}^{2}}.
\end{equation}
\end{thm}
\begin{remark}
Under the same assumption of Theorem \ref{thm: main}, it follows directly from \eqref{eq: mainestimate} that
\begin{equation}\label{eq: almostmain}
\|e^{it\Delta}P_{Q_{1}}fe^{it\Delta}P_{Q_{2}}g\|_{L_{x,t}^{p}(\mathbb{R}\times \mathbb{T}\times [0,T_{p}])}\lesssim_{\varepsilon} R^{\frac{2p-4}{p}+\varepsilon}\|f\|_{L_{x}^{2}}\|g\|_{L_{x}^{2}}.
\end{equation}
And indeed, for \eqref{eq: almostmain}, we can also cover the endpoint $p=\frac{5}{3}$.\\

But it should be noted that \eqref{eq: almostmain} cannot imply \eqref{eq: mainestimate} and it, in some senses, loses derivative and is not very useful when one studies $L^{2}$-critical nonlinear Schr\"odinger equations (NLS) problems\footnote{We refer to \cite{dodson2012global} for the notion of $L^{2}$-criticality where the long time behavior for $L^{2}$-critical NLS is under studied.}.
\end{remark}

One may observe $T_{p}$ in Theorem \ref{thm: main} equals $R^{-1}$ when $p=\frac{5}{3}$, and goes to $1$ as $p\rightarrow 2$. The two endpoints are both special. When $p=\frac{5}{3}$ and $T_{p}=\frac{1}{R}$, the estimate falls in the so-called self similar regime, and estimate \eqref{eq: mainestimate} is expected to follow from the Euclidean case, \cite{Tao2003} (since within this time scale the Schr\"odinger wave does not quite see the periodic effect of $\mathbb{T}$), and the associated estimate has been studied in \cite{Deng2023}, but in the spirit of \cite{TVV1998}, that is
\begin{equation}\label{eq:bilestDeng}
\|e^{it\Delta}P_{Q_{1}}fe^{it\Delta}P_{Q_{2}}g\|_{L_{x,t}^{2}(\mathbb{R}\times \mathbb{T}\times [0,R^{-1}])}\lesssim R^{\frac{2q-4}{q}}\|\widehat{f}\|_{L^{q}(\bR\times \bZ)}\|\widehat{g}\|_{L^{q}(\bR\times \bZ)}, \quad q=\frac{12}{7}.
\end{equation}

Up to $\varepsilon$-loss, the above estimate can be deduced by interpolating \eqref{eq: almostmain} for the $p=\frac{5}{3}$ and an $L^1\times L^1 \to L^{\infty}$ estimate. It should be noted that \eqref{eq:bilestDeng} does not suffer from the $\varepsilon$-loss, and is more useful within the semiclassical time scale $T\lesssim \frac{1}{R}$.

 When $p=2$, Theorem \ref{thm: main} follows directly from the Strichartz estimate on waveguide $\mathbb{R}\times \mathbb{T}$, \cite{TT2001} (see also \cite{barron2020global}).

The \textbf{main point} of our estimate \eqref{eq: mainestimate} is that $p$ can be smaller than $2$, and $T_{p}$ can be pushed to $1$ as $p$ approaches $2$, and in particular, $T_{p}\gg \frac{1}{R}$ for all $p>\frac{5}{3}$. We also note that estimate \eqref{eq: mainestimate} is \textbf{sharp} in the sense that one can construct examples which saturate this estimate. See Appendix B for these examples.

\subsection{Background, motivations and previous works}
To explain the background explicitly, one needs to consider \eqref{eq: lsrt} also on $\mathbb{R}^{2}$ (Euclidean space), and $\mathbb{T}^{2}$ (periodic space, i.e. 2D tori). In this subsection, we will use $e^{it\Delta_{\mathbb{R}\times \mathbb{T}}}$, $e^{it\Delta_{\mathbb{R}^{2}}}$, $e^{it\Delta_{\mathbb{T}^{2}}}$ to denote linear Schr\"odinger operators on different manifolds to avoid unnecessary confusions.

Strichartz estimates for those linear models play a fundamental role in the study of corresponding nonlinear problems.
In $\mathbb{R}^{2}$, see, \cite{KT1998,Cazenave2003}, it reads as
\begin{equation}\label{eq: strr}
\|e^{it\Delta_{\bR^{2}}}f\|_{L_{x,t}^{4}(\bR^{2}\times \bR)}\lesssim \|f\|_{L^{2}(\bR^2)}.
\end{equation}
The parallel estimate on $\mathbb{T}^{2}$ (rational or irrational), see \cite{Bourgain1993,BD2015}, reads as
\begin{equation}\label{eq: sttt}
\|e^{it\Delta_{\bT^{2}}}P_{N}f\|_{L_{x,t}^{4}(\bT^{2}\times [0,1])}\lesssim_{\varepsilon} N^{\varepsilon}\|f\|_{L^{2}(\bT^2)}.
\end{equation}
Estimate \eqref{eq: sttt} is weaker than \eqref{eq: strr}, but optimal\footnote{See Appendix in \cite{TT2001}.}. Note that it only holds local in time and loses $\varepsilon$ derivative. And in particular, it makes the study of 2D cubic (mass-critical) NLS ($iu_{t}+\Delta u=+|u|^{2}u$) at $L^{2}$ regularity on torus out of reach at the moment. Meanwhile, global well-posedness (GWP) and scattering of (defocusing) mass-critical NLS on Euclidean space is one landmark result in dispersive PDEs, see \cite{dodson2012global,dodson2015global,dodson2016global}.

The parallel estimate on $\mathbb{R}\times \mathbb{T}$ can still only hold local in time, but with no loss of derivative, \cite{TT2001},
\begin{equation}\label{eq: strt}
\|e^{it\Delta_{\bR\times \bT}}f\|_{L_{x,t}^{4}(\bR\times \bT \times [0,1])}\lesssim \|f\|_{L^{2}(\bR \times \bT)}.
\end{equation}
In some sense, this makes $\mathbb{R}\times \mathbb{T}$ an interesting manifold to study mass-critical NLS, and the associated Strichartz estimate which does \textbf{not lose} derivative.

Barron \cite{Barron2021} proves global-in-time Strichartz estimates for Schr\"odinger equations on product spaces $\mathbb{R}^n \times \mathbb{T}^m$ with an additional $\varepsilon$-derivative loss via the $\ell^2$-decoupling method. This loss is significant at the critical endpoint of the Strichartz estimates\footnote{Interestingly, Barron shows that this $\varepsilon$-loss can be removed for exponents away from the critical endpoint. This removal enhances the strength of the estimates for those specific exponents.}. These Strichartz estimates are then applied to prove small-data-scattering at the scaling critical regularity for certain nonlinear NLS models. Moreover, Barron-Christ-Pausader's recent work \cite{barron2020global} establishes global-in-time\footnote{Compared with \cite{TT2001}, the estimate in \cite{barron2020global} is global-in-time.} Strichartz estimates for Schr\"odinger equations on product spaces $\mathbb{R} \times \mathbb{T}$, overcoming derivative losses.

It is also natural to study the bilinear version of estimates \eqref{eq: strr},\eqref{eq: sttt},\eqref{eq: strt}.

One direction is to consider $\|e^{it\Delta}P_{N}fe^{it\Delta}P_{M}g\|_{L_{x,t}^{2}}$, i.e. one exploits extra smallness when $f,g$ are localized at different (dyadic) frequencies $N,M$. It indicates direct applications for low regularity GWP results of NLS, we refer to \cite{Bourgain1998,CGSY2023,DFYZZ2024,DPST2007,FSWW2018,ZZ2021}\footnote{We also refer to the Introduction in \cite{DFYZZ2024} for a brief introduction for the topic ``NLS on waveguide manifolds'' (see also the references therein), which has been intensively studied in recent decades. We do not emphasize this point since we focus on the bilinear estimates in this paper.}.

Another direction is to study bilinear restriction estimates as in \cite{Tao2003}. And in $\mathbb{R}^{2}$, it reads as
\begin{equation}\label{eq: taobi}
\|e^{it\Delta_{\mathbb{R}^{2}}}P_{Q^{'}_{1}}fe^{it\Delta_{\mathbb{R}^{2}}}P_{Q^{'}_{2}}g\|_{L_{x,t}^{p}(\bR^2 \times \bR)}\lesssim \|f\|_{L_x^{2}}\|g\|_{L_x^{2}}, \quad \frac{5}{3}< p< 2.
\end{equation}
Here $Q^{'}_{1},Q^{'}_{2}$ are two cubes of scale $1$, and they are separated by distance $1$.

Estimate \eqref{eq: taobi} is of great importance in the study of modern harmonic analysis, see \cite{Tao2001,Tao2003,TVV1998,Wolff2001}. And it is directly related to the concentration compactness theory in dispersive PDEs, see \cite{Bourgain1998,bahouri1999high,kenig2006global,kenig2006global,merle1998compactness}. And one may refer to \cite{BV2007} for a systematic approach to transfer estimate \eqref{eq: taobi} into a profile decomposition.

By the scaling symmetry of $\mathbb{R}^{2}$, estimate \eqref{eq: taobi} is equivalent to
\begin{equation}\label{eq: taobirescale}
\|e^{it\Delta_{\bR^2}}P_{Q_{1}}fe^{it\Delta_{\bR^2}}P_{Q_{2}}g\|_{L_{x,t}^{p}(\bR^2 \times \bR)}\lesssim R^{\frac{2p-4}{p}}\|f\|_{L_{x}^{2}}\|g\|_{L_{x}^{2}}.	
\end{equation}
Here $Q_{1},Q_{2}$ are two cubes of scale $R\ge 1$, and they are separated by distance $R$.

The goal of the current article is to study the analogous estimate of \eqref{eq: taobirescale} in waveguide $\mathbb{R}\times \mathbb{T}$, and hopefully it can be helpful to further applied to implement the scheme of \cite{IPS2012,ionescu2012energy} to transfer the (large data) GWP results in the Euclidean setting \cite{Dodson2}. for mass-critical NLS to the waveguide $\mathbb{R}\times \mathbb{T}$ case.

\subsection{More results and some preliminary reductions}
Theorem \ref{thm: main} estimates bilinear interactions of Schr\"odinger wave localizing at frequency $|\xi|\sim R$ up to time scale $T_{p}:=R^{\frac{4p-8}{3-p}}$. One may wonder what can be said for time scale $T_{p}\leq T\leq 1$. Via interpolation Theorem \ref{thm: main} with (the rather straightforward) bound
\begin{equation}
\|e^{it\Delta}fe^{it\Delta}g\|_{L_{x,t}^{1}(\mathbb{R}\times \mathbb{T}\times [0,T])}\lesssim T \|f\|_{L_{x}^{2}}\|g\|_{L_{x}^{2}},
\end{equation}
one obtains the following corollary,
\begin{coro} \label{T:Full bilinear}
 Let $Q_{1}, Q_{2}$ be two cubes of length $R\ge 1$ on  $\mathbb{R}\times \mathbb{Z}$, and we assume that they are $R$ separated. Then for $p\in (\frac{5}{3},2]$, $T\in [T_p,1]$, and all $f,g\in L^{2}(\mathbb{R}\times \mathbb{T})$, there holds
\[\|e^{it\Delta}P_{Q_{1}}fe^{it\Delta}P_{Q_{2}}g\|_{L^p(\bR \times \bT \times [0,TR^{-\varepsilon}])} \lesssim_\varepsilon T^{\frac{3-p}{2p}} \|f\|_{L_{x}^{2}}\|g\|_{L_{x}^{2}}, \quad \forall \varepsilon>0.\]
For $p\in (\frac{5}{3},2]$, $T\in [R^{-2}, T_p]$, there holds
\[\|e^{it\Delta}P_{Q_{1}}fe^{it\Delta}P_{Q_{2}}g\|_{L^p(\bR \times \bT \times [0,TR^{-\varepsilon}])} \lesssim_\varepsilon R^{\frac{2p-4}{p}} \|f\|_{L_{x}^{2}}\|g\|_{L_{x}^{2}}, \quad \forall \varepsilon>0.\]
For $p \in [1, \frac{5}{3}]$, $T \in [R^{-1},1]$, there holds
\[\|e^{it\Delta}P_{Q_{1}}fe^{it\Delta}P_{Q_{2}}g\|_{L^p(\bR\times \bT \times[0,T  ])} \lesssim_\varepsilon T^{\frac{3-p}{2p}-\varepsilon} \|f\|_{L_{x}^{2}}\|g\|_{L_{x}^{2}}, \quad \forall \varepsilon>0.\]
For $p \in [1, \frac{5}{3}]$, $T \in [R^{-2},R^{-1}]$, there holds
\[\|e^{it\Delta}P_{Q_{1}}fe^{it\Delta}P_{Q_{2}}g\|_{L^p(\bR\times \bT \times[0,T])} \lesssim_\varepsilon T^{\frac{5p-3}{2p}-\varepsilon} R^{\frac{1-p}{p}} \|f\|_{L_{x}^{2}}\|g\|_{L_{x}^{2}}, \quad \forall \varepsilon>0.\]
\end{coro}

\begin{remark} The above estimate is sharp up to $\varepsilon$-loss, see Appendix B. Our methods can also be applied to get similar results for general waveguide $\bR^m\times \bT^n$($m,n\ge 1$), with some suitable modifications.
\end{remark}

\begin{remark}
As mentioned in \cite{Tao2003}, it is expected to apply Corollary \ref{T:Full bilinear} to obtain a \textbf{pointwise convergence} result for the Schr\"odinger equation on $\bR\times \bT$ in a standard way. There has been intensive studies for the pointwise convergence in Euclidean situation, and we refer to \cite{Bourgain2016,DGL2017,DZ2019} for some classical results. We thank Prof. Jiqiang Zheng for helpful communications on this point.
\end{remark}

We reformulate \eqref{eq: mainestimate} in the formal of extension operators. The study of free Schr\"odinger solution is directly related to the restriction problem in harmonic analysis, see for instance \cite{Guth2023} and the reference therein. This is because the space-time Fourier transform of the linear Schr\"odinger equation is supported in the paraboloid. And that is why estimate \eqref{eq: taobi} are usually called bilinear restriction estimates in the literature. Let extension operator $\tilde{E}$ be defined as
\begin{equation}
(\tilde{E}h)(x, t):=\int_{\mathbb{R}^{2}}h(\xi)e^{ix\xi + it|\xi|^{2}} \ddd\xi.
\end{equation}
Here $x\in \mathbb{R}^{2}, \xi\in \mathbb{R}^{2}$. Note that one has $e^{-it\Delta_{\bR^2}}f\equiv \tilde{E}\widehat{f}$.\\

Let $Q$ be a cube in $\mathbb{R}^{2}$, one naturally define $\tilde{E}_{Q}h$ as $\tilde{E}(\chi_{Q}h)$.

Let $Q_{1}'$, $Q_{2}'$ be as in \eqref{eq: taobi}, then \eqref{eq: taobi} is equivalent to
\begin{equation}
\|E_{Q_{1}'}h_{1}E_{Q_{2}'}h_{2}\|_{L_{x, t}^{p}}\lesssim \|h_{1}\|_{L^{2}_{\xi}}\|h_{2}\|_{L_{\xi}^{2}}.
\end{equation}
And one may also reformulate \eqref{eq: taobirescale} in a similar fashion.

We now set up the extension operator in the setting of waveguide, i.e. we define the extension opearator $E$ as

\begin{equation}\label{eq: extension}
Eh(x,t):= \int_{\bR\times \bZ} e^{ix\xi+it|\xi|^2} h(\xi) \ddd \xi \equiv \sum_{\xi_2\in \bZ} \int_{\bR} e^{i(x_1\xi_1+x_2\xi_2) + it(\xi_1^2 + \xi_2^2)} h(\xi_1,\xi_2) \ddd\xi_1,
\end{equation}
where $x\in \bR\times \bT$, $t\in \bR$ and $h\in L^2(\bR \times \bZ)$.
Strictly speaking, we should use $e^{2\pi ix \xi}$ rather than $e^{ix\xi}$ in \eqref{eq: extension}, we neglect this issue for notation simplicity.

And one defines $E_{Q}h$ as $E(\chi_{Q}h)$.

Now, let $Q_{1}, Q_{2}$ be two cubes of scale $R$, and separated by distance $\sim R$ as in Theorem \ref{thm: main}, the main estimate \eqref{eq: mainestimate} can be \textbf{rewritten} as
\begin{equation} \label{T:Key bilinear-11}
		\l\|E_{Q_1} h_1 E_{Q_2} h_2 \r\|_{L_{x,t}^p(\bR\times \bT \times [0,T_{p}R^{-\varepsilon}])} \lesssim R^{\frac{2p-4}{p}} \|h_1\|_{L^2(\bR\times \bZ)} \|h_2\|_{L^2(\bR \times \bZ)},
	\end{equation}
where $h_1, h_2\in L^2(\bR\times \bZ)$. We now focus on estimate \eqref{T:Key bilinear-11} in the rest of the article.

\subsection{Structure of the article}
The rest of the article is structured as follows. We end this section with some notations which will be used later. In Section \ref{S:Overview}, we recall some classical Euclidean bilinear estimates of Tao \cite{Tao2003}, and then state the outline of our bilinear estimate on waveguide manifold. In Section \ref{S:Proof of main results}, we present the detailed proof of our main result Theorem \ref{thm: main}. Finally, we show a waveguide version epsilon-removal lemma in Appendix A, and then show the aforementioned sharpness examples in Appendix B.

\subsection{Notations} \label{SubS:Notations}
We write $A \lesssim B$ to say that there is a constant $C$ such that $A\leq CB$. We use $A \sim B$ when $A \lesssim B \lesssim A $. Particularly, we write $A \lesssim_u B$ to express that $A\leq C(u)B$ for some constant $C(u)$ depending on $u$.

Then we give some more preliminaries in the setting of waveguide manifold. Throughout this paper, we regularly refer to the spacetime norms
\begin{equation}
    \|u\|_{L^q_xL^p_t(\mathbb{R}^m\times \mathbb{T}^n \times I)} :=\left(\int_{I}\left(\int_{\mathbb{R}^m\times \mathbb{T}^n} |u(t,x)|^q dx \right)^{\frac{p}{q}} dt\right)^{\frac{1}{p}}.
\end{equation}
Moreover, we turn to the Fourier transformation and Littlewood-Paley theory. We define the Fourier transform on $\mathbb{R}^m \times \mathbb{T}^n$ as follows:
\begin{equation}
    \widehat{f}(\xi) := \int_{\mathbb{R}^m \times \mathbb{T}^n}f(x)e^{-ix\xi} \ddd x,
\end{equation}
where $\xi=(\xi_1,\xi_2,...,\xi_{d})\in \mathbb{R}^m \times \mathbb{Z}^n$ and $d=m+n$. We also note the Fourier inversion formula
\begin{equation}
    f(x)=c \sum_{(\xi_{m+1},...,\xi_{d})\in \mathbb{Z}^n} \int_{(\xi_1,...,\xi_{m}) \in \mathbb{R}^m} \widehat{f}(\xi)e^{ix \xi} \ddd\xi_1\ldots\ddd\xi_m.
\end{equation}
Moreover, we define the Schr{\"o}dinger propagator $e^{it\Delta}$ by
\begin{equation}
    \widehat{e^{it\Delta}f} (\xi) :=e^{it|\xi|^2}\widehat{f}(\xi).
\end{equation}
We are now ready to define the Littlewood-Paley projections. First, we fix $\eta_1: \mathbb{R} \rightarrow [0,1]$, a smooth even function satisfying
\begin{equation}
    \eta_1(\xi) =
\begin{cases}
1, \ |\xi|\le 1,\\
0, \ |\xi|\ge 2,
\end{cases}
\end{equation}
and $N=2^j$ a dyadic integer. Let $\eta^d=\mathbb{R}^d\rightarrow [0,1]$, $\eta^d(\xi)=\eta_1(\xi_1)\eta_1(\xi_2)\eta_1(\xi_3)...\eta_1(\xi_d)$. We define the Littlewood-Paley projectors $P_{\leq N}$ and $P_{ N}$ by
\begin{equation}
    \widehat{P_{\leq N} f}(\xi):=\eta^d\left(\frac{\xi}{N}\right) \widehat{f}(\xi), \quad \xi \in \mathbb{R}^m \times \mathbb{Z}^n,
\end{equation}
and
\begin{equation}
P_Nf:=P_{\leq N}f-P_{\leq \frac{N}{2}}f.
\end{equation}
Littlewood-Paley projection $P_{Q}$ can be also defined in a natural way where $Q$ is a cube on $\mathbb{R}\times \mathbb{Z}$.

Let the rescaled torus be defined as  $\bT_{R} := R\bT$ for $R\geq 1$. Correspondingly $\bZ_{1/R}:= R^{-1}\bZ$. We define\[\int_{\bR\times \bZ_{1/R}} h(\xi) \ddd \xi := R^{-1} \sum_{\xi_2 \in \bZ_{1/R}} \int_{\bR} h(\xi_1,\xi_2) \ddd \xi_1,\]
and the rescaled extension operator $E^R$ as follows
\[E^R h(x,t):= \int_{\bR \times \bZ_{1/R}} e^{ix\xi + it |\xi|^2} h(\xi) \ddd \xi.\]
For a subset $Q^{'} \subset \bR \times \bZ_{1/R}$, we introduce the operator $E_{Q^{'}}^R$ as
\[E_{Q^{'}}^R h(x,t):= \int_{Q^{'}} e^{ix\xi + it |\xi|^2} h(\xi) \ddd \xi.\]
Based on these notations, notice that
\[E_{Q} h\l(\frac{x}{R}, \frac{t}{R^2}\r) = RE_{Q^{'}}^R h_R(x,t), \quad h_R(\xi):= Rh(R\xi), \quad Q^{'}:= R^{-1} Q.\]
For a finite set $S$, we use the notation $\# S$ to denote the number of its elements.

\subsection*{Acknowledgment}
B. Di was supported by the Postdoctoral Fellowship Program of CPSF under Grant Number GZB20230812. C. Fan was partially supported by the National Key R\&D Program of China, 2021YFA1000800, CAS Project for Young Scientists in Basic Research, Grant No.YSBR-031, and NSFC grant no.12288201. Z. Zhao was supported by the NSF grant of China (No. 12101046, 12271032) and the Beijing Institute of Technology Research Fund Program for Young Scholars.

\section{An overview}\label{S:Overview}
\subsection{A recall of Tao's bilinear estimates on Euclidean space}
In this subsection, due to the aforementioned connections between free (linear) Schr\"odinger equations and Fourier restriction operators, we recall some of Tao's classical bilinear restriction estimates on Euclidean space (\cite{Tao2003}). There are mainly three steps summarized as follows.

\textbf{Step 1:} As shown in the literature \cite{Tao2003,TVV1998,Wolff2001}, to establish some desired bilinear restriction estimates, one typical method is to establish a local estimate (which may have $\varepsilon$-loss but have endpoint index) and then apply an $\varepsilon$-removal lemma. More precisely, to establish the desired global-type bilinear restriction estimate \eqref{eq: taobi}, it suffices to show that the following local-type bilinear restriction estimate
\begin{equation} \label{E:Tao local bilinear}
	\l\|e^{it\Delta_{\mathbb{R}^{2}}}P_{Q^{'}_{1}}f e^{it\Delta_{\mathbb{R}^{2}}}P_{Q^{'}_{2}}g \r\|_{L_{x,t}^{\frac{5}{3}}(B_r(x_0,t_0))}\lesssim C_{\alpha} r^{\alpha} \|f\|_{L_{x}^{2}}\|g\|_{L_{x}^{2}}
\end{equation}
holds for all spacetime balls $B_r(x_0,t_0)$ of radius $r$ and all $\alpha>0$. Note that this kind of reductions unavoidably lead to some non-endpoint index (non-sharp) phenomena.

\textbf{Step 2:} Next, based on this local-type reduction in the first step, one need to further investigate the local-type bilinear estimate \eqref{E:Tao local bilinear}. Here one typical tool is called \textit{wave packet decomposition}, which can be viewed as a quantity version of the uncertainty principle. Roughly speaking, wave packet decomposition states that the solution $e^{it\Delta_{\mathbb{R}^2}} f$ can be decomposed as a summation of wave packets $\phi_{T}$, which have good localization in both physical space and frequency space. Formally, all these wave packets $\phi_T$ can be viewed as characteristic functions of tubes $T$, which have base radii of $r$ and lengths of $r^2$. More precisely, there exists a wave packet decomposition
\begin{equation*}
	e^{it\Delta_{\mathbb{R}^{2}}}P_{Q^{'}_{1}}f (x)=\sum_{T_1 \in \mathcal{T}_1^{'}} c_{T_1} \phi_{T_1}(x,t),
\end{equation*}
where $\mathcal{T}_1^{'}$ is a collection of tubes that is associated to the cube $Q_1^{'}$, and the coefficients $c_{T_1}$ obey the following $\ell^2$ bound $\|c_{T_1}\|_{\ell^2} \lesssim \|f\|_{L_x^2}$. Similar decomposition also holds for the item $e^{it\Delta_{\mathbb{R}^{2}}}P_{Q^{'}_{2}}g$. Thus by putting the wave packet decomposition into the desired estimate \eqref{E:Tao local bilinear}, one can reduce this desired estimate to a tube-type (wave packet) estimate. Note that these kind of reductions can transfer the integral-type problem into a tube-type (wave packet) problem.

\textbf{Step 3:} Moreover, based on this \textit{tube-type reduction} in the second step, one need to further investigate the associated \textit{tube-type estimate}. To do so, under certain transverse condition, Tao has established an incidence estimate of tubes, which is his main conclusion in \cite{Tao2003}. More precisely, to show the desired local bilinear restriction estimate \eqref{E:Tao local bilinear}, Tao established an incidence estimate on wave packets, which claims that the following estimate
\begin{equation} \label{E:Tao bilinear-wave packets 5/3}
	\l\|\sum_{T_1\in \mathcal{T}_1^{'}} \sum_{T_2\in \mathcal{T}_2^{'}} \phi_{T_1} \phi_{T_2}\r\|_{L_{x,t}^{\frac{5}{3}}([0,r]^3)} \lesssim C_{\alpha} r^{\alpha} \l(\# \mathcal{T}_{1,[0,r]^3}^{'} \r)^{1/2} \l(\# \mathcal{T}_{2,[0,r]^3}^{'} \r)^{1/2}
\end{equation}
holds for all radii $r>0$ and all parameters $\alpha>0$. Here the collections $\mathcal{T}_1^{'}$ and $\mathcal{T}_2^{'}$ satisfy certain transverse condition, due to the geometric property of paraboloid; and the notation $\mathcal{T}_{i, [0,r]^3}^{'}$ are defined as
\[\mathcal{T}_{i, [0,r]^3}^{'} := \{T: T\in \mathcal{T}_i^{'}, T\cap [0,r]^3 \neq \emptyset\}.\]
We point out that this incidence estimate is highly nontrivial, and is proved by induction on scales with the help of certain combinatorial argument. In this article, we will use this incidence-type estimate as a \textit{blackbox}.

We \textbf{emphasize} that the aforementioned wave packet decomposition is very important for us, since it helps us connect Schr\"odinger wave on $\bR^2$ and on $\bR\times \bT$, through incidence-type estimate. This intuition will be achieved later in the proof of our main result.

\subsection{Overview of our proof}
Here we show the outline of our proof, which basically follows the aforementioned Euclidean steps. However, some special phenomena have appeared due to our waveguide $\bR\times \bT$ setting. Thus there need some necessary adjustments to handle these phenomena.

First, we establish a waveguide version $\varepsilon$-removal lemma, which will lead to the $\varepsilon$-loss in our main estimate \eqref{eq: mainestimate}. In summary, applying the local-type reduction in the waveguide setting, we not only lose the endpoint Lebesgue space index but also lose the $R^{-\varepsilon}$ time regime. That is why we have $\varepsilon$-loss in our main Theorem \ref{thm: main}. To state the waveguide version $\varepsilon$-removal argument more precisely, we naturally introduce the cubes $B_r$ in rescaled-waveguide $\bR \times \bT_{R} \times \bR$ as follows
\[B_r:=[0,r]\times [0,r] \times [0,r] \;\;\; \text{for} \;\; r\leq R, \quad B_r:=[0,r]\times [0,R] \times [0,r] \;\;\;\text{for}\;\; r>R.\]
Then, by the waveguide version $\varepsilon$-removal lemma and some rescaled arguments, we can reduce our desired global-type estimate \eqref{T:Key bilinear-11} to the following local-type bilinear estimate
\[\l\|E_{Q^{'}_1}^R h_1 E_{Q^{'}_2}^R h_2 \r\|_{L_{x,t}^p\l(B_r\r)} \lesssim C_{\alpha} r^{\alpha} \|h_1\|_{L^2(\bR\times \bZ_{1/R})} \|h_2\|_{L^2(\bR \times \bZ_{1/R})}, \quad \forall r\in [1,R^{\frac{2p-2}{3-p}}].\]
Note that the waveguide cube $B_r$ is different from the classical Euclidean cubes when $r>R$. This fact prevents us from directly applying Tao's classical incidence estimate. We will deal with this difference later, by making use of certain periodic property.

Second, we establish a waveguide version wave packet decomposition, which has some periodic properties due to the waveguide $\bR\times \bT$ setting. Roughly speaking, the waveguide version wave packet decomposition satisfies all the properties of Euclidean version wave packet decomposition, and further satisfies one more periodic property in the second spatial direction due to the half-periodic manifold $\bR\times \bT$. More precisely, there exists a decomposition
\begin{equation*}
	E_{Q_1^{'}}^{R} h_1 =\sum_{T_1} \widetilde{c}_{T_1} \psi_{T_1},
\end{equation*}
where $T_1$ ranges over a collection that is associated to the cube $Q_1^{'}$, and the coefficients $\widetilde{c}_{T_1}$ obey the $\ell^2$ bound $\|\widetilde{c}_{T_1}\|_{\ell^2} \lesssim \|h_1\|_{L_\xi^2}$. Moreover, for each fixed $T_1$, the function $\psi_{T_1}$ can be viewed as a characteristic function of the following periodic-type tube
\[\bigcup_{m\in\bZ} \l[T_{1}+ (0,mR,0) \r],\]
where each translated tube $T_{1}+ (0,mR,0)$ has base radial $r$ and length $r^2$. Similar decomposition also holds for the item $E_{Q_2^{'}}^{R} h_2$. Hence, we can reduce our desired local-type bilinear estimate to a tube-type (wave packet) estimate, and the localized waveguide spacetime cube $B_r$. As mentioned above, to apply Tao's classical tube-type bilinear estimate, it remains to transfer the waveguide cube $B_r$ to the standard Euclidean cube $[0,r]^3$.

Finally, we use the aforementioned periodic property to extend the data, and transfer our $B_r$-type wave packet estimate to the standard $[0,r]^3$-type wave packet estimate. This periodic argument allows us to apply Tao's incidence estimate in our situation, and we hence complete the proof.

\section{Proof of main results} \label{S:Proof of main results}
\subsection{Local estimate: integral version}
In this subsection, we do the local-type reduction. Recall that we are going to show the desired bilinear estimate \eqref{T:Key bilinear-11}. By rescaling, it suffices to show
\[\l\|E_{Q_1^{'}}^R h_1 E_{Q_2^{'}}^R h_2 \r\|_{L^p(\bR\times \bT_R \times [0,T_p R^{2-\varepsilon}])} \lesssim \|h_1\|_{L^2(\bR\times \bZ_{1/R})} \|h_2\|_{L^2(\bR \times \bZ_{1/R})}.\]
It is natural to introduce the waveguide cubes $B_r$ in $\bR \times \bT_{R} \times \bR$ as follows
\[B_r:=[0,r]\times [0,r] \times [0,r] \;\;\; \text{for} \;\; r\leq R, \quad B_r:=[0,r]\times [0,R] \times [0,r] \;\;\;\text{for}\;\; r>R.\]
Then by the waveguide version $\varepsilon$-removal lemma, which will be presented in the Appendix A, it suffices to show the following local estimate.
\begin{prop} \label{P:Local estimate-integral}
	Given $p\in [\frac{5}{3},2]$ and $R\gg 1$. For the dyadic cubes
	\[Q^{'}_1 \in R^{-1}\mathcal{Q}_R, \quad Q^{'}_2\in R^{-1}\mathcal{Q}_R, \quad \dist(Q^{'}_1,Q^{'}_2) \gg 1,\]
	the following bilinear estimate
	\begin{equation} \label{P:Rescaled key bilinear-1}
		\l\|E_{Q^{'}_1}^R h_1 E_{Q^{'}_2}^R h_2 \r\|_{L_{x,t}^p\l(B_r\r)} \lesssim_{\alpha} r^{\alpha} \|h_1\|_{L^2(\bR\times \bZ_{1/R})} \|h_2\|_{L^2(\bR \times \bZ_{1/R})},
	\end{equation}
	holds for all dyadic numbers $r\in [1,R^{\frac{2p-2}{3-p}}]$.
\end{prop}

\subsection{Wave packet decomposition and semiclassical regime}
Similar to the Euclidean case, we need a waveguide version wave packet decomposition to investigate the local estimate \eqref{P:Rescaled key bilinear-1}. In this subsection, we are going to establish the desired waveguide version wave packet decomposition, which can also help us directly handle the semiclassical regime, see Remark \ref{R:Simiclassical range} below. The classical Euclidean version wave packet decomposition can be seen in \cite[Lemma 4.1]{Tao2003}, see also \cite[Section 2]{Demeter2020}.

For a base set $ Q^{'}_i \subset \bR \times \bZ_{1/R}$ and the associated cap
\[S_i:= \l\{(\xi,|\xi|^2): \xi \in Q^{'}_i \r\}, \quad S_i \subset \bR\times \bZ_{1/R} \times \bR,\]
denote a $S_i$-tube (with scale $r$) to be any set of the form
\[T:= \{(x,t): \frac{r}{2}\leq t \leq r, |x-(x_T + tv_T)| \leq r^{\frac{1}{2}}\},\]
with
\[x_T \in r^{\frac{1}{2}}\bZ^2, \quad v_T \in r^{-\frac{1}{2}}\bZ^2 \cap Q^{'}_i, \quad (1,v_T) \bot S_i.\]
For convenience, we further denote the single-period $\widetilde{S}_i$-tube (with scale $r$) to be any set of the form
\[T:= \{(x,t):\frac{r}{2}\leq t \leq r, |x-(x_T + tv_T)| \leq r^{\frac{1}{2}}\},\]
with
\[x_T \in r^{\frac{1}{2}}\bZ^2 \cap (\bR \times [0,R]), \quad v_T \in r^{-\frac{1}{2}}\bZ^2 \cap Q^{'}_i, \quad (1,v_T) \bot S_i.\]
\begin{lemma}[Waveguide wave packet decomposition at scale $r$] \label{L:Waveguide wave packet}
	Let $1\leq \sqrt{r} \leq R$ and the sets $Q_i^{'} \subset \bR \times \bZ_{1/R}$. For each $h_i \in L^2(Q^{'}_i)$, there exists a decomposition
	\begin{equation} \label{L:Waveguide wave packet-1}
		E_{Q_i^{'}}^{R} h_i(x,t) =\sum_{T_i} c_{T_i} \phi_{T_i} (x,t)
	\end{equation}
	where $T_i$ ranges over all $S_i$-tubes. This decomposition satisfies the following properties: for each fixed tube $T_i$ and time $t\in[\frac{r}{2},r]$, the function $\phi_{T_i}(t)$ has Fourier support in the set
	\[\{\xi\in \bR^2: \xi = v_{T} + O(r^{-\frac{1}{2}})\}\]
	and obeys the pointwise decay estimates
	\begin{equation} \label{L:Waveguide wave packet-2}
		|\phi_{T_i}(x,t)| \lesssim r^{-\frac{1}{2}} \l(1+ \frac{|x-(x_T +tv_T)|}{r^{\frac{1}{2}}}\r)^{-N}, \quad\quad \forall N>0;
	\end{equation}
	then for any collection $\mathcal{T}_i$ of $S_i$-tubes, there holds the following probability estimate
	\begin{equation} \label{L:Waveguide wave packet-3}
		\l\|\sum_{T_i \in \mathcal{T}_i} \phi_{T_i} (x,t)\r\|_{L^2}^2 \lesssim \# \mathcal{T}_i;
	\end{equation}
	and the complex-valued coefficients $c_{T_i}$ obey the periodic property
	\begin{equation} \label{L:Waveguide wave packet-4}
		|c_{T_i}|= |c_{T_{i,m}}|, \quad T_{i,m}:= T_i + (0, mR, 0) \quad \text{with} \quad m\in \bZ;
	\end{equation}
	moreover, by periodization, if we rewrite this decomposition \eqref{L:Waveguide wave packet-1} as follows
	\begin{equation} \label{L:Waveguide wave packet-5}
		E_{Q_i^{'}}^{R} h_i(x,t) =\sum_{T_i} \widetilde{c}_{T_i} \psi_{T_i}(x,t), \quad\quad \psi_{T_i}(x,t):= e^{ixv_T+ it|v_T|^2} \sum_{m\in \bZ} \phi_{T_i,m}(x,t),
	\end{equation}
	where $T_i$ ranges over all $\widetilde{S}_i$-tubes, then the single-period coefficients $\widetilde{c}_{T_i}$ obey the $\ell^2$ bound
	\begin{equation} \label{L:Waveguide wave packet-6}
		\|\widetilde{c}_{T_i}\|_{\ell^2} \lesssim \|h_i\|_{L^2}.
	\end{equation}
\end{lemma}

\begin{remark} \label{R:Simiclassical range}
	By the waveguide wave packet decomposition Lemma \ref{L:Waveguide wave packet}, for the desired local estimate Proposition \ref{P:Local estimate-integral}, the case $r\lesssim R$ can be viewed as Euclidean case. The key observation is that there are only $O(1)$ single-period $\widetilde{S}_i$-tubes in the decomposition \eqref{L:Waveguide wave packet-1}, and all the Euclidean properties similarly hold on waveguide.
\end{remark}

\begin{proof}[\textbf{Proof of Lemma \ref{L:Waveguide wave packet}}]
	We imitate the proof for classical Euclidean version wave packet decomposition, which can be seen in \cite[Section 2]{Demeter2020}. Consider a smooth function $\Upsilon: [-4,4]^2 \to [0,\infty)$ such that
	\[\sum_{\eta\in \bZ^2}\Upsilon(\xi-\eta) = 1, \quad \forall\; \xi\in\bR^2.\]
	Then for each tube $T$ with the position $x_T \in r^{\frac{1}{2}}\bZ^2$ and velocity $v_T \in r^{-\frac{1}{2}}\bZ^2$, we define the associated function
	\[\psi_T(x,t):= r^{-\frac{1}{2}} e^{ixv_T + it |v_T|^2} \int_{\bR\times \bZ_{\sqrt{r}/R}} e^{\frac{i(x-x_T +t v_T) \xi}{\sqrt{r}} + \frac{it |\xi|^2}{r}} \Upsilon (\xi) \ddd \xi,\]
	and the associated coefficient
	\[\widetilde{c}_T:= r^{-\frac{1}{2}} e^{-i x_T v_T} \int_{\bR \times \bZ_{1/R}} e^{-ix_T \xi} h(\xi) \Upsilon(r^{\frac{1}{2}}(\xi-v_T)) \ddd\xi.\]
	Note that $\psi_{T}(x_1,x_2,t) = \psi_{T}(x_1,x_2+m R,t)$ up to the constant $2\pi$ in the phase function. Denote the direction $e_2:=(0,1,0)$. By the Poisson summation formula, one can obtain
	\[\psi_T(x,t)= e^{ixv_T + it |v_T|^2} \sum_m \phi_{T_{m}}(x,t), \quad \phi_{T_m}:= r^{-\frac{1}{2}} \int_{\bR^2} e^{\frac{i(x-x_T +t v_T -mRe_2) \xi}{\sqrt{r}} + \frac{it |\xi|^2}{r}} \Upsilon (\xi) \ddd \xi.\]
	Then by the partition of unit and Fourier series theory, we obtain the desired wave packet decomposition \eqref{L:Waveguide wave packet-1} and \eqref{L:Waveguide wave packet-5}. Applying the non-stationary phase estimate and the $\bR\times \bT_R$ version of Parseval identity, one can check all these desired properties. Hence the proof is finished.
\end{proof}

Notice the following implication
\[r\leq R^{\frac{2p-2}{3-p}} \quad \Rightarrow \quad r^{\frac{5-3p}{2p}} (r R^{-1})^{\frac{1}{p'}} \leq 1.\]
Thus to show Proposition \ref{P:Local estimate-integral}, recalling the Remark \ref{R:Simiclassical range}, it suffices to show the following lemma.
\begin{lemma} \label{L:Local estimate-integral}
	Given $p\in [\frac{5}{3},2]$ and $R\gg 1$. For the dyadic cubes
	\[Q_1 \in R^{-1}\mathcal{Q}_R, \quad Q_2\in R^{-1}\mathcal{Q}_R, \quad \dist(Q_1,Q_2) \gg 1,\]
	we have the following bilinear estimate
	\begin{equation} \label{L:Rescaled key bilinear-1}
		\l\|E_{Q_1}^R h_1 E_{Q_2}^R h_2 \r\|_{L_{x,t}^p\l(B_r\r)} \lesssim_{\alpha} r^{\alpha} r^{\frac{5-3p}{2p}} (r R^{-1})^{\frac{1}{p'}} \|h_1\|_{L^2(\bR\times \bZ_{1/R})} \|h_2\|_{L^2(\bR \times \bZ_{1/R})}, \quad \forall r\in [R,R^{\frac{2p-2}{3-p}}].
	\end{equation}
\end{lemma}

\subsection{Local estimate: tube version}
In this subsection, we reduce the desired local bilinear integral estimate \eqref{L:Rescaled key bilinear-1} to a tube-type incidence estimate. Note that the waveguide version wave packet decomposition satisfies the periodic property \eqref{L:Waveguide wave packet-5} and single-period bound \eqref{L:Waveguide wave packet-6}. Therefore, we only need to consider the periodic collections $\mathcal{T}_i$ of $S_i$-tubes, which means
\[T\in \mathcal{T}_i \quad \Rightarrow \quad T+(0,mR,0) \in \mathcal{T}_i, \quad \forall m\in\bZ.\]
For a subset $B\subset \bR^3$, we introduce the intersection notation
\[\mathcal{T}_{i, B} := \{T: T\in \mathcal{T}_i, T\cap B \neq \emptyset\}.\]
Applying the waveguide version wave packet decomposition Lemma \ref{L:Waveguide wave packet} and some dyadic pigeonholing arguments as shown in \cite[page 1369-1370]{Tao2003}, to establish the desired bilinear estimate \eqref{L:Rescaled key bilinear-1}, it suffices to establish the following single-period incidence estimate
\begin{equation} \label{E:Tube bilinear estiamte-single}
	\l\|\sum_{T_1\in \mathcal{T}_1} \sum_{T_2\in \mathcal{T}_2} \phi_{T_1} \phi_{T_2} \r\|_{L^p(B_r)} \lessapprox r^{\frac{5-3p}{2p}} (r R^{-1})^{\frac{1}{p'}} \l(\# \mathcal{T}_{1,B_r} \r)^{1/2} \l(\# \mathcal{T}_{2,B_r} \r)^{1/2}, \quad \forall r\in [R,R^{\frac{2p-2}{3-p}}]
\end{equation}
Here the notation $A\lessapprox B$ means that $A\leq C_{\alpha} r^{\alpha} B$ holds for all $\alpha>0$.

At the end of this subsection, we point out that our waveguide problem has been reduced to an Euclidean problem now. This is because that the wave packets $\phi_{T_i}$ can be formally viewed as characteristic function of Euclidean tubes, and the waveguide cubes $B_r$ are indeed Euclidean cuboids (all these objects are in the Euclidean situation now).

\subsection{Periodic property and the bilinear estimate}
As mentioned above, by applying the waveguide version wave packet decomposition, one can reduce a waveguide problem to an associated Euclidean problem. Therefore, based on this reduction, one can apply the classical Euclidean-type bilinear restriction estimate to establish an associated waveguide-type bilinear estimate.

However, compared with the Euclidean case, one can see that the waveguide cubes $B_r$ appeared in \eqref{E:Tube bilinear estiamte-single} is different from the standard cubes in $\bR^3$. To overcome this problem, we make use of the aforementioned periodic property of $\mathcal{T}_i$, so that we can extend the integral range $B_r$ to the range $[0,r]^3$. Indeed, by the periodic property of $\mathcal{T}_i$, to show the desired single-period incidence estimate \eqref{E:Tube bilinear estiamte-single}, it suffices to establish the following multiple-period incidence estimate
\begin{equation} \label{E:Tube bilinear estiamte}
	\l\|\sum_{T_1\in \mathcal{T}_1} \sum_{T_2\in \mathcal{T}_2} \phi_{T_1} \phi_{T_2} \r\|_{L^p([0,r]^3)} \lessapprox r^{\frac{5-3p}{2p}} \l(\# \mathcal{T}_{1,[0,r]^3} \r)^{1/2} \l(\# \mathcal{T}_{2, [0,r]^3} \r)^{1/2}, \quad \forall r \in [R,R^{\frac{2p-2}{3-p}}].
\end{equation}

Next we will explain that this desired estimate essentially follows from Tao's bilinear estimate. In his paper \cite{Tao2003}, Tao mainly states his result for the crucial index $p=\frac{5}{3}$, which indeed comes from interpolation with the $L^1$-estimate and $L^2$-estimate. Here we restate Tao's bilinear restriction result for general index $p\in [1,2]$. In fact, Tao's bilinear estimate on wave packets can be stated as follows
\begin{equation} \label{E:Tao bilinear-wave packets}
	\l\|\sum_{T_1\in \mathcal{T}_1^{'}} \sum_{T_2\in \mathcal{T}_2^{'}} \phi_{T_1} \phi_{T_2}\r\|_{L^p([0,r]^3)} \lessapprox r^{\frac{5-3p}{2p}} \l(\# \mathcal{T}_{1,[0,r]^3}^{'} \r)^{1/2} \l(\# \mathcal{T}_{2,[0,r]^3}^{'} \r)^{1/2}, \quad \forall r \in [R,R^{\frac{2p-2}{3-p}}].
\end{equation}
Here $\mathcal{T}_i^{'}$ is an arbitrary collection of $S_i$-tubes (may be not periodic), and the notation $\mathcal{T}_{i, [0,r]^3}^{'}$ is parallelly defined as
\[\mathcal{T}_{i, [0,r]^3}^{'} := \{T: T\in \mathcal{T}_i^{'}, T\cap [0,r]^3 \neq \emptyset\}.\]
Finally, it can be seen that Tao's bilinear estimate on wave packets \eqref{E:Tao bilinear-wave packets} directly gives our desired estimate \eqref{E:Tube bilinear estiamte}, and thus we complete the proof.

\section*{Appendix A: Waveguide version epsilon-removal lemma}
Here we show a sketch proof of the $\varepsilon$-removal lemma in our waveguide setting. To begin with, we introduce the rescaled waveguide cubes $B_r^R$ in $\bR \times [0,1] \times [0,T]$ as follows
\[B_r^R:=[0,rR^{-1}]\times [0,rR^{-1}] \times [0,rR^{-2}] \;\;\; \text{for} \;\; r\leq R, \quad B_r^R:=[0,rR^{-1}]\times [0,1] \times [0,rR^{-2}] \;\;\;\text{for}\;\; r>R.\]
We consider $r \in [1, R^2T]$ to be dyadic numbers. Then our $\varepsilon$-removal lemma can be stated as follows.
\begin{lemma} \label{L:Epsilon-removal}
	For $p_0\in [\frac{5}{3},2)$ and $R^{-2} \lesssim T \leq 1$, if there holds the local bilinear estimate
	\begin{equation} \label{E:Local bilinear}
		\l\|E_{Q_1^{*}} h_1 E_{Q_2^{*}}h_2 \r\|_{L_{x,t}^{p_0}(B_r^R)} \lesssim_{\alpha} r^{\alpha} R^{2-\frac{4}{p_0}} \|h_1\|_{L^2(\bR\times \bZ)} \|h_2\|_{L^2(\bR\times \bZ)}, \quad\quad \forall r\in [1,R^2 T], \;\; \forall \alpha>0,
	\end{equation}
	then there holds the following global estimate
	\[\l\|E_{Q_1} h_1 E_{Q_2} h_2 \r\|_{L_{x,t}^p(\bR\times \bT \times [0,T])} \lesssim R^{2-\frac{4}{p}} \|h_1\|_{L^2(\bR\times \bZ)} \|h_2\|_{L^2(\bR\times \bZ)}, \quad \forall p\in(p_0,2].\]
\end{lemma}

\begin{proof}[\textbf{Proof of Lemma \ref{L:Epsilon-removal}}]
	The proof mainly relies on two facts: the Fourier transform of waveguide-paraboloid measure has some decay, and the $L^2$-$L^4$ Strichartz estimate on the waveguide $\bR\times \bT$ (see \cite{TT2001}). Here we adapt some ideas from \cite[Lemma 2.4]{TV2000-I}.
	
	Since the estimate
	\[\l\|E_{Q_1^{*}} h_1 E_{Q_2^{*}}h_2 \r\|_{L_{x,t}^{\infty}(\bR\times \bT \times [0,T])} \lesssim  R^{2} \|h_1\|_{L^2(\bR\times \bZ)} \|h_2\|_{L^2(\bR\times \bZ)}\]
	holds, it suffices to show the following weak type estimate
	\[|B| \lesssim \lambda^{-p} R^{-4}, \quad\quad B:= \l\{(x,t) \in \bR\times \bT \times [0,T]: \l|E_{Q_1} h_1 E_{Q_2} h_2\r|> \lambda R^2 \r\}\]
	holds for all $0<\lambda \lesssim 1$ and $\|h_1\|_{L^2} = \|h_2\|_{L^2} = 1$. Moreover, it suffices to consider the case $|B|\gtrsim R^{-4}$. By following the arguments in \cite[Proof of Lemma 2.4]{TV2000-I}, for fixed $(\lambda, h_1, h_2)$ which means fixed $B$, it suffices to show the following bilinear estimate
	\[\l\|\chi_B E_{Q_1} g_1 E_{Q_2} g_2\r\|_{L^1_{x,t}} \lesssim R^{2-\frac{4}{p}} |B|^{1-\frac{1}{p}} \|g_1\|_{L^2} \|g_2\|_{L^2}, \quad \forall g_1 \in L^2, \quad \forall g_2\in L^2.\]
	Similarly, by applying some duality arguments, for fixed $g_2$ with $\|g_2\|_{L^2}\sim 1$ , it suffices to show
	\begin{equation} \label{E:Epsilon-removal-5}
		\l\langle \widetilde{F} \ast E\chi_{Q_1}, \widetilde{F}\r\rangle_{L_{x,t}^2} \lesssim R^{4-\frac{8}{p}} |B|^{2-\frac{2}{p}}, \quad \|F\|_{L^{\infty}} \lesssim 1, \quad \widetilde{F}:= \chi_B \overline{E_{Q_2} g_2} F.
	\end{equation}
	Next we divide the time interval into three regions
	\[I_1:= \l\{t: |t| \gg T \r\}, \quad I_2:= \l\{t: |t|\gg M \r\}, \quad I_3:= \l\{t: |t|\lesssim M, |t|\lesssim T \r\}.\]
	Here the constant $M$ will be determined later. Choose a non-negative Schwartz function $\phi$ satisfying
	\[\phi:\bR\to \bR, \quad \supp \phi \subset [-1,1], \quad \phi\equiv 1 \;\;\text{on}\;\; [-\frac{1}{2},\frac{1}{2}], \quad \widehat{\phi} \gtrsim 1 \;\; \text{on} \;\; [-1,1].\]
	The fact that $\supp(\widetilde{F}) \subset \bR\times \bT\times [0,T]$ implies that the first region $I_1$ can be computed as follows
	\[\l\langle \widetilde{F} \ast \l(E\chi_{Q_1}(1- \phi(\frac{t}{2^{10}T}))\r), \widetilde{F}\r\rangle_{L_{x,t}^2} = 0.\]
	Then we consider the second region $I_2$, where we need two facts: the scale-invariant Strichartz estimate and the Fourier transform of surface measure has some decay. First, the Strichartz estimate on $\bR\times \bT$ and H\"older's inequality implies that
	\[\|\widetilde{F}\|_{L^1} \lesssim \|F\|_{L^{\infty}} \|\chi_B\|_{L^{\frac{4}{3}}} \|\chi_{\bR\times \bT\times [0,T]}E_{Q_2} g_2\|_{L^4} \lesssim |B|^{ \frac{3}{4}}.\]
	Second, the Van der Corput lemma on the first variable $\xi_1$ can imply
	\[\|E\chi_{Q_1} (1-\phi(\frac{t}{M}))\|_{L^{\infty}} \lesssim M^{-\frac{1}{2}} R.\]
	Note that our decay in $\bR\times \bT$ is $M^{-\frac{1}{2}}$ rather than $M^{-1}$ in the Euclidean case. However, as we will see later, this decay is enough to establish the desired epsilon-removal consequence. Combining these aforementioned two estimates, we conclude
	\[\l\langle \widetilde{F} \ast \l(E\chi_{Q_1}(1- \phi(\frac{t}{M}))\r), \widetilde{F}\r\rangle_{L_{x,t}^2} \lesssim |B|^{\frac{3}{2}} M^{-\frac{1}{2}} R.\]
	Recalling the desired estimate \eqref{E:Epsilon-removal-5}, we can choose $M:= |B|^{4/p-1} R^{16/p-6}$ and then consider the third region $I_3$. Here we introduce the notation
	\[r_0:=\min\l\{2^{10} TR^2, (|B|R^4)^{4/p-1} \r\},\]
	 it remains to show the following
	 \begin{equation} \label{E:Epsilon-removal-6}
	 	\langle \widetilde{F}\ast \l(E 1_{Q_1} \phi(\frac{t}{r_0 R^{-2}})\r) ,\widetilde{F} \rangle_{L^2} \lesssim R^{4-\frac8p} |B|^{2-\frac2p}.
	 \end{equation}
	Next we are going to use the assumption \eqref{E:Local bilinear} with $r\sim r_0$ to prove this result. Notice that the index $p$ in the right hand side of \eqref{E:Epsilon-removal-6} is strict lager than $p_0$. Hence, we can take a suitable $\alpha$ in \eqref{E:Local bilinear} to show the desired estimate \eqref{E:Epsilon-removal-6}. The details are shown as follows.
	
	Let $\mathscr{F}$ denote the space-time Fourier transform for $(x,t)$. Define
	\[\Gamma_{i,\gamma}=\{(\xi,\tau): \xi \in Q_i, |r_0 R^{-2}(\tau-|\xi^2|)| \sim \gamma\}, \quad i=1,2.\]
	For $(\xi,\tau)\in \Gamma_{1,\gamma}$,
	$$ |\mathscr{F} \l(E 1_{Q_1} \phi(\frac{t}{r_0 R^{-2}}) \r)(\xi,\tau)|=r_0 R^{-2} 1_{Q_1}(\xi) |\widehat{\phi}(r_0 R^{-2}(\tau-|\xi|^2))| \lesssim_N r_0 R^{-2} \gamma^{-N}.$$
	So
	\begin{align*}
		\langle \widetilde{F}\ast \l(E 1_{Q_1} \phi(\frac{t}{r_0 R^{-2}})\r) ,\widetilde{F} \rangle_{L^2}&\lesssim \int |\mathscr{F}(\widetilde{F})|^2 |\mathscr{F} \l(E 1_{Q_1} \phi(\frac{t}{r_0 R^{-2}})\r) |\\
		&\lesssim_N \sum_{\gamma_1} \int_{\Gamma_{1,\gamma_1}}|\mathscr{F}(\widetilde{F})|^2  r_0 R^{-2} \gamma_1^{-N},
	\end{align*}
	it suffices to show that
	$$ \l(\int_{\Gamma_{1,\gamma_1}}|\mathscr{F}(\widetilde{F})|^2 \r)^{\frac12} \lesssim \gamma_1^N r_0^{-1/2} R^{3-\frac4p} |B|^{1-\frac1p}$$
	for all $\gamma_1\ge 1$ and some $N\ge 1$. From the definition of $\widetilde{F}$, this will follow if we show
	$$ \|1_B E_{Q_2} g_2 F\|_{L^2(\Gamma_{1,\gamma_1})} \lesssim \gamma_1^N r_0^{-1/2} R^{3-\frac4p} |B|^{1-\frac1p} \|F\|_{L^\infty}\|g_2\|_{L^2},$$
	for all $F,g_2$. By duality, it suffices to prove
	$$ \|1_B \mathscr{F}(G_1) E_{Q_2} g_2\|_{L^1} \lesssim \gamma_1^N r_0^{-1/2} R^{3-\frac4p} |B|^{1-\frac1p} \|G_1\|_{L^2_{\xi,\tau}} \|g_2\|_{L^2_\xi}$$
	for all $g_2$ and all $G_1$ supported on $\Gamma_{1,\gamma_1}$.
	
	We repeat the similar argument on $g_2$ when we fix $G_1$, so it suffices to show that
	$$ \|1_B \mathscr{F}(G_1)\mathscr{F}(G_2)\|_{L^1}\lesssim (\gamma_1 \gamma_2)^N r_0^{-1} R^{4-\frac4p} |B|^{1-\frac1p}  \|G_1\|_{L^2_{\xi,\tau}} \|G_2\|_{L^2_{\xi,\tau}}  $$
	for all $\gamma_i \ge 1$, $G_i$ supported on $\Gamma_{i,\gamma_i}$, $i=1,2$ and some $N\ge1$.
	By the H{\"o}lder inequality and $B\subset A$, it suffices to show that
	$$  \|1_B \mathscr{F}(G_1)\mathscr{F}(G_2)\|_{L^{p_0}}\lesssim (\gamma_1 \gamma_2)^N r_0^{-1} R^{4-\frac4p} |B|^{\frac1{p_0}-\frac1p}  \|G_1\|_{L^2_{\xi,\tau}} \|G_2\|_{L^2_{\xi,\tau}}.$$
	
	Choose $r\in [1,T R^2]$ such that $r\sim r_0$. We break the proof into two cases.
	
	\emph{Case 1:} there holds $r<R$. Choose a function $\psi_r:\bZ \to \bR$ satisfying
	$$ \supp \psi_r \subset [-r^{-1}R,r^{-1}R]\cap \bZ, \quad \psi_r\equiv 1 \;\;\text{on}\;\; [-\frac{r^{-1}R}{2},\frac{r^{-1}R}{2}], \quad \widehat{\psi_r} \gtrsim 1 \;\; \text{on} \;\; [-rR^{-1},rR^{-1}].$$
	So
	$$\supp\l(\mathscr{F}^{-1}(\widehat{\phi}(\frac{x_1-\widetilde{x_1}}{r R^{-1}})\widehat{\psi_r}(x_2-\widetilde{x_2})\widehat{\phi}(\frac{t-\widetilde{t}}{r R^{-2}})  \mathscr{F}(G_i)     ) \r) \subset N_{\gamma_i r^{-2} R^2}(S_i^*), \quad i=1, 2,$$
	for any $(\widetilde{x_1},\widetilde{x_2},\widetilde{t})\in \bR \times \bT \times \bR $. Here we have used the notations
	\begin{equation} \label{E:Epsilon-removal-9}
		S_i^*:=\{(\xi,|\xi|^2):\xi\in Q_i^{*}\}, \quad i=1,2,
	\end{equation}
	and
	\begin{equation} \label{E:Epsilon-removal-10}
		N_{\gamma}(S_i^*):=\{(\xi,\tau)\in (\bR\times \bZ)\times \bR:  \xi \in Q_i^{*}, |\tau-|\xi|^2|\lesssim \gamma \}, \quad i=1,2.
	\end{equation}
	By the Cauchy-Schwarz inequality, the Lemma \ref{lem:thick bilinear with alpha} stated below and the fact $B\subset \bR\times \bT\times [0,T]$, we conclude
	\begin{align*}
		&\|1_B \mathscr{F}(G_1)\mathscr{F}(G_2)\|_{L^{p_0}} \le \sum_{(\widetilde{x_1},\widetilde{x_2},\widetilde{t})} \| \mathscr{F}(G_1)\mathscr{F}(G_2)\|_{L^{p_0}\l(B_r^R+(\widetilde{x_1},\widetilde{x_2},\widetilde{t})\r)} \\
		&\lesssim \sum_{(\widetilde{x_1},\widetilde{x_2},\widetilde{t})} \| \l(\widehat{\phi}(\frac{x_1-\widetilde{x_1}}{r R^{-1}})\widehat{\psi_r}(x_2-\widetilde{x_2})\widehat{\phi}(\frac{t-\widetilde{t}}{r R^{-2}})\r)^2  \mathscr{F}(G_1)\mathscr{F}(G_2)\|_{L^{p_0}\l(B_r^R+(\widetilde{x_1},\widetilde{x_2},\widetilde{t})\r)} \\
		&\lesssim_\alpha \sum_{(\widetilde{x_1},\widetilde{x_2},\widetilde{t})}  R^{2-\frac{4}{p_0}} r^\alpha \prod_{i=1}^{2} \l( (\gamma_i r^{-2} R^2)^{\frac12}  \|\widehat{\phi}(\frac{x_1-\widetilde{x_1}}{r R^{-1}})\widehat{\psi_r}(x_2-\widetilde{x_2})\widehat{\phi}(\frac{t-\widetilde{t}}{r R^{-2}})  \mathscr{F}(G_i)\|_{L^2} \r)  \\
		&\lesssim R^{2-\frac{4}{p_0}} r^\alpha (\gamma_1 \gamma_2)^{\frac12} r^{-1} R^2 \|G_1\|_{L^2} \|G_2\|_{L^2} ,
	\end{align*}
	where the $(\widetilde{x_1},\widetilde{x_2},\widetilde{t})$ is in
	$$\l\{\l(k_1 r R^{-1},k_2 r R^{-1}, k_3 r R^{-2} \r) : k_1,k_2,k_3 \in \bZ,\;\; 0\le k_2 \lesssim \frac{1}{r R^{-1}}, \;\; 0\le k_3 \lesssim \frac{T}{r R^{-2}} \r\}.$$
	Recalling the assumption $|B|\gtrsim R^{-4}$, we will be done if
	$$ r^{\alpha} \sim r_0^\alpha \lesssim (|B|R^4)^{\frac{2}{p_0}-\frac2p}.$$
	From the definition of $r_0$, we only need to choose $\alpha< (\frac{2}{p}-\frac12)^{-1} (\frac{1}{p_0}-\frac1p)$.
	
	\emph{Case 2:} There holds $r\ge R$. Note that
	$$\supp\l(\mathscr{F}^{-1}(\widehat{\phi}(\frac{x_1-\widetilde{x_1}}{r R^{-1}})\widehat{\phi}(\frac{t-\widetilde{t}}{r R^{-2}})  \mathscr{F}(G_i)    ) \r) \subset N_{\gamma_i r^{-2} R^2}(S_i^*), \quad i=1, 2,$$
	and the rest of the discussion is similar.
\end{proof}

Recalling the notations in \eqref{E:Epsilon-removal-9} and \eqref{E:Epsilon-removal-10}, it remains to prove the following thick-bilinear restriction estimate.
\begin{lemma}\label{lem:thick bilinear with alpha}
	Given $\gamma_1,\gamma_2>0$. Suppose the local bilinear estimate \eqref{E:Local bilinear} holds, then we have
	$$ \l\| \mathscr{F}^{-1} (F_1) \mathscr{F}^{-1} (F_2) \r\|_{L^{p_0}(B_r^R)}\lesssim_\alpha  {R}^{2-\frac{4}{p_0}} r^{\alpha} \gamma_1^{\frac12} \gamma_2^{\frac12} \|F_1\|_{L^2(N_{\gamma_1}(S_1^*))} \|F_2 \|_{L^2(N_{\gamma_2}(S_2^*))} $$
	holds for all $F_i$ supported on $N_{\gamma_i}(S_i^*)$, $i=1,2$, where $\mathscr{F}^{-1}$ is the space-time inverse Fourier transform.
\end{lemma}
\begin{proof}
	Direct computations can give that
	\begin{align*}
		\mathscr{F}^{-1} (F_i)(x,t)&=\int_{N_{\gamma_i}(S_i^*)}e^{ix\xi +it \tau} F_i(\xi,\tau) \ddd \xi \ddd \tau  \\
		&=\int_{-\gamma_i}^{\gamma_i} \int_{Q_i^{*}} e^{ix\xi+it|\xi|^2} e^{it\tau_i} F_i(\xi,|\xi|^2+\tau_i) \ddd \xi \ddd \tau_i \\
		&= \int_{-\gamma_i}^{\gamma_i}e^{it\tau_i} (E_{Q_i^{*}} h_{i,\tau_i})(x,t) \ddd \tau_i,
	\end{align*}
	where $h_{i,\tau_i}(\xi)=F_i(\xi,|\xi|^2+\tau_i)$, $i=1,2$. Then by the Cauchy-Schwarz inequality, the Minkowski inequality and the assumption \eqref{E:Local bilinear}, we have
	\begin{align*}
		\| \mathscr{F}^{-1}(F_1)   \mathscr{F}^{-1}(F_2)\|_{L^{p_0}(B_r^R)} &= \l\|  \int_{-\gamma_1}^{\gamma_1} \int_{-\gamma_2}^{\gamma_2} e^{it(\tau_1+\tau_2)} E_{Q_1^{*}} h_{1,\tau_1} E_{Q_2^{*}} h_{2,\tau_2}\ddd \tau_2    \ddd \tau_1     \r\|_{L^{p_0}(B_r^R)} \\
		& \le \int_{-\gamma_1}^{\gamma_1} \int_{-\gamma_2}^{\gamma_2} \|E_{Q_1^{*}} h_{1,\tau_1} E_{Q_2^{*}} h_{2,\tau_2}\|_{L^{p_0}( B_r^R)}\ddd \tau_2    \ddd \tau_1 \\
		&\lesssim_\alpha {R}^{2-\frac{4}{p_0}} r^{\alpha} \int_{-\gamma_1}^{\gamma_1} \int_{-\gamma_2}^{\gamma_2} \|h_{1,\tau_1} \|_{L^2(Q_1^{*})} \|h_{2,\tau_2} \|_{L^2(Q_2^{*})}\ddd \tau_2    \ddd \tau_1  \\
		&\lesssim_\alpha {R}^{2-\frac{4}{p_0}} r^{\alpha} \gamma_1^{\frac12} \gamma_2^{\frac12} \|F_1\|_{L^2(N_{\gamma_1}(S_1^*))} \|F_2 \|_{L^2(N_{\gamma_2}(S_2^*))}.
	\end{align*}
	This completes the proof.
\end{proof}
\begin{remark}
	For higher dimensional waveguide $\bR^m \times \bT^n$ with $m\geq 1$, the corresponding epsilon-removal result still holds by applying some noncritical scale-invariant Strichartz estimates, see for instance \cite{Barron2021,KV2016}.
\end{remark}

\section*{Appendix B: Some examples for explaining the sharpness}
Here we give some examples to show that our conclusion Corollary \ref{T:Full bilinear} is sharp up to the $\varepsilon$-loss. These three examples are shown as follows.

Let $C(R,T,p)$ be the sharp constant that
\[\|E_{Q_{1}}h_1 E_{Q_{2}}h_2\|_{L_{x,t}^p(\bR\times \bT \times[0,T])} \lesssim C(R,T,p)  \|h_1\|_{L^{2}}\|h_2\|_{L^{2}}, \quad \forall h_1,h_2 \in L^2(\bR\times \bZ).\]
We only consider the situation $T\ge R^{-2}$.

First, we choose the following
\[Q_1= [0,R] \times \{0,1,\cdots, R\}, \quad Q_2= [0,R] \times \{100R,100R+1,\cdots, 101R\}, \quad h_1=\chi_{Q_1}, \quad h_2=\chi_{Q_2},\]
then
\[|E_{Q_1}h_1(x,t)|\gtrsim R^{2}, \quad |E_{Q_2}h_2(x,t)|\gtrsim R^{2}, \quad \forall (x,t) \in \l\{|x|\lesssim R^{-1}, |t|\lesssim R^{-2} \r\}.\]
These implies the constant $C(R,T,p)$ satisfies
\[C(R,T,p) \gtrsim R^{\frac{2p-4}{p}}.\]
Second, we choose the following
\[Q_1=[0,T^{-1/2}] \times \{0\}, \quad Q_2=[0,T^{-1/2}] \times \{100R\}, \quad h_1=\chi_{Q_1}, \quad h_2=\chi_{Q_2},\]
then
\[|E_{Q_1}h_1(x,t)|\gtrsim T^{-1/2}, \quad |E_{Q_2}h_2(x,t)|\gtrsim T^{-1/2}, \quad \forall (x,t) \in \l\{|x_1|\lesssim T^{1/2}, x_2\in \bT, |t|\lesssim T\r\}.\]
These implies the constant $C(R,T,p)$ satisfies
\[C(R,T,p) \gtrsim T^{\frac{3-p}{2p}}.\]
Third, we choose the following
\[Q_1=[0,a]\times \l([0,a^2R^{-1}]\cap \bZ\r), \quad Q_1=[0,a]\times \l([100R, 100R+a^2R^{-1}]\cap \bZ\r), \quad h_1=\chi_{Q_1}, \quad h_2=\chi_{Q_2},\]
where $a=\max\{R^{1/2}, T^{-1/2}\}$, then
\[|E_{Q_1}h_1(x,t)|\gtrsim a^3R^{-1}, \quad |E_{Q_2}h_2(x,t)|\gtrsim a^3R^{-1}, \quad \forall (x,t) \in S_a,\]
where the set $S_a$ is defined as
\[S_a:= \l\{|x_1| \lesssim a^{-1}, |x_2|\lesssim a^{-2}R, |x_2+200 Rt|\lesssim a^{-2} R, |t|\lesssim a^{-2}\r\}.\]
These implies
\[C(R,T,p) \gtrsim (R+ T^{-1})^{\frac{3p-5}{2p}} R^{\frac{1-p}{p}}.\]
In summary, we have proved the sharpness of our Corollary \ref{T:Full bilinear} up to the $\varepsilon$-loss.

\vspace{30mm}

\newcommand{\etalchar}[1]{$^{#1}$}

\begin{flushleft}
	\vspace{0.3cm}\textsc{Yangkendi Deng\\
		Academy of Mathematics and Systems Science, Chinese Academy of Sciences, Beijing, 100190, People's Republic of China} \\
	\emph{E-mail address}: \textsf{dengyangkendi@amss.ac.cn}
	
	\vspace{0.3cm}\textsc{Boning Di \\
		Academy of Mathematics and Systems Science, Chinese Academy of Sciences, Beijing, 100190, People's Republic of China} \\
	\emph{E-mail address}: \textsf{diboning@amss.ac.cn}
	
	\vspace{0.3cm}\textsc{Chenjie Fan \\
		Academy of Mathematics and Systems Science, Chinese Academy of Sciences, Beijing, 100190, People's Republic of China \\
		Hua Loo-Keng Key Laboratory of Mathematics, Chinese Academy of Sciences, Beijing, 100190, People's Republic of China} \\
	\emph{E-mail address}: \textsf{fancj@amss.ac.cn}
	
	\vspace{0.3cm}\textsc{Zehua Zhao \\
		Department of Mathematics and Statistics, Beijing Institute of Technology, Beijing, People's Republic of China \\
		Key Laboratory of Algebraic Lie Theory and Analysis of Ministry of Education, Beijing, People's Republic of China} \\
	\emph{E-mail address}: \textsf{zzh@bit.edu.cn}
\end{flushleft}

\end{document}